\documentclass[12pt]{amsart}
\usepackage{amsmath,amssymb}
\usepackage{latexsym}
\usepackage{amsthm}
\usepackage{lineno}
\usepackage[normalem]{ulem}
\makeatletter
\@namedef{subjclassname@2020}{%
	\textup{2020} Mathematics Subject Classification}
\makeatother

\title[Generically Multiply Transitive Actions]{Groups of finite Morley rank with a generically multiply transitive action on an abelian group}
\author{Ay\c{s}e Berkman}
\address{Mathematics Department, Mimar Sinan Fine Arts University, Istanbul, Turkey}
\email{ayse.berkman@msgsu.edu.tr, ayseberkman@gmail.com}
\author{Alexandre Borovik}
\address{Department of Mathematics, University of Manchester, UK}
\email{alexandre $\gg {\rm at} \ll $ borovik.net}
\thanks{\copyright\ 2021 Ay\c{s}e Berkman and Alexandre Borovik}

\date{10 February 2022}

\dedicatory{Dedicated to Tuna Alt\i nel in celebration of his freedom}

\begin{document}
\newtheorem{theorem}{Theorem}
\newtheorem{problem}{Problem}
\newtheorem{conjecture}{Conjecture}
\newtheorem{lemma}{Lemma}[section]
\newtheorem{corollary}[lemma]{Corollary}

\newtheorem{proposition}[lemma]{Proposition}

\newtheorem{remark}[lemma]{Remark}
\newtheorem{fact}[lemma]{Fact}
\newtheorem{question}[lemma]{Question}
\theoremstyle{definition}
\newtheorem*{definition}{Definition}
\newtheorem*{notation}{Notation}
\newcommand{\acf}{algebraically closed field }
\newcommand{\acfd}{algebraically closed field}
\newcommand{\acfs}{algebraically closed fields}
\newcommand{\fmr}{finite Morley rank }
\newcommand{\rk}{\operatorname{rk}}
\newcommand{\pgl}{\operatorname{PGL}}
\newcommand{\ddeg}{\operatorname{deg}}
\newcommand{\psrk}{{\rm psrk}}
\newcommand{\stab}{{\rm stab}}
\newcommand{\fmrd}{finite Morley rank}
\newcommand{\bi}{\begin{itemize}}
\newcommand{\ei}{\end{itemize}}
\newcommand{\bd}{\begin{description}}
\newcommand{\ed}{\end{description}}
\newcommand{\bF}{\mathbb{F}}

\begin{abstract}
We investigate the configuration where a group of finite Morley rank acts definably and generically $m$-transitively on an elementary abelian $p$-group of Morley rank $n$, where $p$ is an odd prime, and $m\geqslant n$. We conclude that $m=n$, and the action is equivalent to the natural action of $\operatorname{GL}_n(F)$ on $F^n$ for some algebraically closed field $F$. This strengthens one of our earlier results, and partially answers two problems posed by Borovik and Cherlin in 2008. \end{abstract}

\maketitle

\section{Introduction}


This paper is the fourth, and the concluding work in a series of papers; the previous ones being  \cite{bbgeneric,bbpseudo,bbsharp}. All were aimed at
proving the following theorem, but they handled different stages of the proof, each using a completely different approach and technique.

\begin{theorem} Let $G$ be a group of \fmrd, $V$ an elementary abelian $p$-group of Morley rank $n$, and $p$ an odd prime. Assume that $G$ acts on $V$ faithfully, definably and generically $m$-transitively with $m \geqslant n$. Then $m=n$ and there is an \acf $F$ such that $V\simeq F^n$, $G\simeq\operatorname{GL}_n(F)$, and the action is the natural action. \label{maintheorem}
\end{theorem}

In~\cite{bbsharp}, the same theorem was proven under the extra assumption that the action of $G$ on $V$ is generically
{\em sharply} $m$-transitive. In this paper, we prove the generic sharpness of the action of $G$ on $V$  under the hypothesis of Theorem~\ref{maintheorem}. Then  Theorem~\ref{maintheorem} follows from the previous result  \cite[Theorem 1]{bbsharp}. We use the technique developed in \cite{borovik2020} for analysis of actions of certain subgroups of $G$ specifically for the needs of the present project, see Section \ref{subsec:Maschke}.

Theorem \ref{maintheorem} gives partial confirmations to the following two conjectures; note that the latter is implicit in the former.

\begin{conjecture} \cite[Problem 37, p. 536]{abc}, \cite[Problem 13]{borche} Let $G$ be a connected group of finite Morley rank
	acting faithfully, definably, and generically $n$-transitively on a connected
	abelian group $V$ of Morley rank $n$. Then $V$ has a structure of an  $n$-dimensional vector space over an algebraically closed field $F$ of Morley
	rank 1, and $G$ is $\operatorname{GL}_n(F)$ in its natural action on $F^n$.
	\end{conjecture}

\begin{conjecture}  \cite[Problem 12]{borche}
Let $G$ be a connected group of finite Morley rank
	acting faithfully, definably, and generically $t$-transitively on an
	abelian group $V$ of Morley rank $n$. Then $t\leqslant n$.
\end{conjecture}

 The cases when $V$ is a torsion free abelian group or an elementary abelian $2$-group require completely different approaches and  methods and are handled in our next paper. But even that result will not be the end of the story, since it appears to be an almost inevitable step in any proof of the following conjecture.

\begin{conjecture} \cite[Problem 36, p. 536]{abc}, \cite[Problem 9]{borche} \label{mainconjecture}
Let $G$ be a connected group of finite Morley rank acting faithfully, definably, transitively, and generically $(n+2)$-transitively on a set $\Omega$ of Morley rank $n$. Then the pair $(G,\Omega)$ is equivalent to the projective linear group $\pgl_{n+1}(F)$ acting on the projective space $\mathbb{P}^n(F)$ for some algebraically closed field $F$.
\end{conjecture}

Indeed, the group $F^n\rtimes \operatorname{GL}_n(F)$ is the stabiliser of a point in the action of $\pgl_{n+1}(F)$ on  $\mathbb{P}^n(F)$.

Alt\i nel and Wiscons have already made important contributions towards a solution to the above conjecture in \cite{altwis} and \cite{altwis2}. The importance of Conjecture \ref{mainconjecture} has been recently highlighted in \cite{Freitag-Moosa2021}.

General discussion and a survey of results on actions of groups of finite Morley rank can be found in \cite{Borovik-Deloro2019}.
Terminology and notation follow \cite{abc}, \cite{bn} and \cite{borche}.

\section{Useful Facts}

In what follows, $(G,X)$ is an infinite permutation group of \fmrd.

\begin{definition} Let $Y$ be a definable subset of $X$. If $\rk(X\setminus Y)<\rk(X)$ then $Y$ is called a \textit{strongly generic} subset of $X$. We will simply call it a \textit{generic subset}. If $G$ acts transitively on a generic subset of $X$, then we say $G$ acts \textit{generically transitively} on $X$.
If the induced action of $G$ on $X^n$ is generically transitive, then we say 	$G$ acts \textit{generically $n$-transitively} on $X$.\end{definition}

The following two facts show that connectedness assumptions are superfluous in our context.

\begin{fact}\label{connectednessofV}

If $G$ acts generically $m$-transitively on a group $X$, where $m\geqslant\rk(X)$, then $X$ is a connected group.
\end{fact}

\begin{proof}
 If $m\geqslant 2$ then, this is a special case of \cite[Lemma 1.8]{borche}. When $m=1$, note that the generic orbit, say $A\subseteq X$, is cofinite in $X$. Since $G$ fixes $X^\circ$ and $A$ setwise, $G$ also fixes $X^\circ\cap A$ setwise. The transitivity of $G$ on $A$ implies $A\subseteq X^\circ$. Hence $X=X^\circ$, since $A$ is cofinite.
\end{proof}

\begin{fact}\label{connectedG}{\rm \cite[Lemma 4.10]{altwis}}
If $G$ acts $n$-transitively on $X$, and $X$ is of degree $1$, then $G^\circ$ also acts $n$-transitively on $X$.
\end{fact}

For any prime $p$, recall that a connected solvable $p$-group of bounded exponent is called a \textit{$p$-unipotent} group, and a divisible abelian $p$-group is called a \textit{p-torus}.

As the following two facts show, the structure of Sylow 2-subgroups  in groups of \fmrd\ is well understood.

\begin{fact}\label{sylow2}{\rm \cite[Propositions I.6.11, I.6.4, I.6.2]{abc}} Sylow $2$-subgroups of a group of \fmr are conjugate. Moreover, if $S$ is a Sylow $2$-subgroup of a group of finite Morley rank, then $S^\circ=U*T$, where $U$ is a definable $2$-unipotent group, and $T$ is a $2$-torus. In particular, Sylow $2$-subgroups in groups of finite Morley rank are locally finite.
\end{fact}

\begin{fact}\label{degenerate}{\rm \cite{bbcdegenerate}, \cite[Theorem IV.4.1]{abc}} Sylow $2$-subgroups of a connected group of \fmr are either trivial or infinite.
\end{fact}

The following is a structure theorem for nilpotent groups of \fmrd.

\begin{fact} \label{structure-of-nilpotent} \cite[Theorem 6.8]{bn}
	Let $G$ be a nilpotent group of finite Morley rank. Then $G$ is
	the central product $D * B$ where $D$ and $B$ are definable characteristic subgroups
	of $G$, $D$ is divisible, $B$ has bounded exponent.
\end{fact}

We gather below some facts about solvable groups of \fmr which will be used in our proof of Theorem \ref{maintheorem}.

\begin{fact}  \label{commutator-nilpotent}  \label{structure-solvable}
	Let $M$ be a connected solvable group of finite Morley rank. Then the following holds.\\
{\rm (a)} The commutator subgroup $[M,M]$ is connected and nilpotent.\\
{\rm (b)} The group $M$ can be written as a product	$M = [M,M] C$, where $C$ is a connected nilpotent subgroup.\\
{\rm (c)} If $M$ is  of bounded exponent, then $M$ is nilpotent.
	\label{solvnilp}
\end{fact}

\begin{proof} They follow from \cite[Theorem 6.8]{bn}, \cite[Corollary I.8.30]{abc}, \cite[Lemma I.5.5]{abc}, respectively.
	\end{proof}

Next, we list some results about various configurations where groups act on groups.

\begin{fact}{\rm \cite[Fact 2.12]{bbsharp}} \label{weightspaces}
Let $V$ be a connected abelian group and $E$ an elementary abelian $2$-group of order $2^m$ acting definably and faithfully on $V$. Assume $m\geqslant n=\rk(V)$ and $V$ contains no involutions.
Then $m = n$ and $V = V_1 \oplus\cdots\oplus V_n$, where
\bi
\item[{\rm (a)}]
every subgroup $V_i$, $i = 1,\dots,n$, is connected, has Morley rank $1$ and is $E$-invariant.
\ei
Moreover,
\bi
\item[{\rm (b)}] each $V_i$,  $i = 1,\dots,n$, is a weight space of $E$, that is, there exists a {non-trivial} homomorphism $\rho_i:E\to \{\pm 1\}$ such that
     \[
     V_i = \{ v \in V \mid v^e = \rho_i(e)\cdot v \mbox{ for all }   e \in E\}.
     \]

\ei
\end{fact}

\begin{proof} Statements can be found in \cite{bbsharp}, whose proofs refer to \cite[Lemma 7.1]{bbpseudo}.
	\end{proof}

Assume that $G$ acts on a group $V$ such that the only infinite definable invariant subgroup of $V$ is itself under this action, then we say $G$ acts on $V$ \textit{minimally}, or $V$ is \textit{$G$-minimal}.

\begin{fact}{\rm \cite[Proposition 2.18]{bbsharp}} \label{smoothlyirreducible}
	Let $V$ be a connected abelian group and $\Sigma=\mathbb Z_2^m \rtimes {\rm Sym}_m$ act definably and faithfully on $V$. Assume $m\geqslant \rk(V)$ and $V$ contains no involutions. Then $\Sigma$ acts on $V$ minimally.
	\end{fact}

\begin{fact} {\rm (Zilber)} {\rm \cite[Theorem 9.1]{bn}}
Let $A$ and $V$ be
	connected abelian groups of \fmr such that $A$ acts on $V$ definably, $C_A(V)=1$ and V is $A$-minimal. Then there exists an algebraically
	closed field $K$ and a definable subgroup $S\leqslant K^*$ such that the action $A\curvearrowright V$ is definably equivalent to the
	natural action of $S$ on $K^+$. \label{zilber}
\end{fact}

\begin{fact}{\rm \cite[Lemma I.8.2]{abc}} Let $G$ be a connected solvable group acting on an abelian group $V$. If $V$ is $G$-minimal, then $G'$ acts trivially on $V$. \label{minimalaction}
\end{fact}

Recall that if a group has no non-trivial $p$-elements, we call it a \textit{$p^\perp$-group}. A connected divisible abelian group is called a \textit{torus}, and a torus $A$ is called \textit{good} if every definable subgroup of $A$ is the definable hull of its torsion elements.

\begin{fact}  \label{fact:good-torus}  \cite[Proposition I.11.7]{abc}
	If a connected solvable $p^\perp$-group $A$ acts faithfully on an abelian $p$-group $V$, then $A$ is a good torus.
\end{fact}

\begin{fact}\label{ptori}{\rm \cite[Proposition I.8.5]{abc}}  Let $p$ be a prime. Assume $V\unlhd G$ is a definable solvable subgroup that contains no $p$-unipotent subgroup, and $U\leqslant G$ is a definable connected $p$-group of bounded exponent, then $[U,V]=1$.
\end{fact}

\begin{fact}{\rm \cite[Lemma I.4.5]{abc}} A definable group of automorphisms of an infinite field of \fmr is trivial.
\label{automorphisms}
\end{fact}


\section{Definable actions on elementary abelian $p$-groups}

In this section, $V$ is a connected elementary abelian $p$-group of finite Morley rank and $X$ is a  finite group acting on $V$ definably. We use additive notation for the group operation on $V$ and treat $V$ as a vector space over $\bF_p$.

It is convenient to work with the ring $R$ generated by $X$ in $\mathop{{\rm End}}V$. It is finite and its elements are definable endomorphisms; $R$ is traditionally called the enveloping algebra (over $\bF_p$) of the action of  $X$ on $V$. We treat $V$ as a right $R$-module.

If $v\in V$,  the set
\[
vR = \{\, vr : r\in R\,\}
\]
is an $R$-submodule, it is called the\emph{ cyclic submodule generated by} $v$. Of course all cyclic submodules contain less than $|R|$ elements and therefore, there are finitely many possibilities for the isomorphism type of each of them.

\subsection{Coprime actions} \label{subsec:Maschke}

In this subsection, we assume that $X$ is a finite $p'$-group acting on $V$. Recall that a torsion group is called a \textit{$p'$-group}, if it has no non-trivial $p$-elements.

We recall some generalities from representation theory. Applying  Maschke's Theorem to the action of $X$ on $R$ by right multiplication, we see that $R$ is a semisimple  $\bF_p$-algebra and that every finite $R$-submodule in $V$ is semisimple, that is, a direct sum of simple modules.

The following important (but easy) result (generalising \cite[Theorem 5]{borovik2020})  now follows immediately.

\begin{theorem} \label{th:Maschke} Let $V$ be a connected elementary abelian $p$-group of finite Morley rank, $X$ a finite $p'$-group acting on $V$ definably, and $R$ the enveloping algebra over $\mathbb F_p$ for the action of $X$ on $V$.
	Assume that  $A_1,A_2, \dots, A_m$ is the complete list of non-trivial simple  submodules for $R$ in $V$, up to isomorphism.
	Then
	\[
	\rk V \geqslant m.
	\]
\end{theorem}

\begin{proof}
	For each $i=1,\ldots,m$, denote
	\[
	V_i = \{\, v\in R :  \mbox{ all simple submodules of  } vR \; \mbox{  are isomorphic to  }\; A_i\,\}.
	\]
	It is easy to see that all $V_i$ are definable submodules of $V$ and
	\[
	V = V_1 \oplus V_2 \oplus\cdots\oplus V_m.
	\]
	Since $V$ is connected, all $V_i$'s are connected.
	Hence, being a non-trivial, definable, connected submodule, each $V_i$ has Morley rank at least $1$. Therefore, $\rk(V)\geqslant m$.
\end{proof}

\begin{problem}
	It would be interesting to remove from Theorem \emph{\ref{th:Maschke}} the assumption that $X$ is a $p'$-group and prove the following:
	\begin{quote}
		If $A_1,A_2, \dots, A_m$ are  nontrivial simple pairwise non-isomorphic $R$-modules appearing as sections $W/U$ for some definable $R$-modules $U < W \leqslant V$, then
		\[
		\rk V \geqslant m.
		\]
	\end{quote}
\end{problem}

\subsection{$p$-Group actions}

The following is folklore, and this elegant and short proof is suggested by the anonymous referee.

\begin{fact}	\label{trivialaction}
	Let $V$ be a connected elementary abelian $p$-group of \fmrd.
	\bi
	\item[{\rm (a)}] If $x$ is a $p$-element and $\langle x\rangle$ acts on $V$ definably, then $[V,x]$ is a proper subgroup of $V$. In particular, if $\rk(V)=1$, then the action is trivial.
	\item[{\rm (b)}] If $P$ is a $p$-torus which acts on $V$ definably, then the action is trivial.
	\ei
\end{fact}

\begin{proof} (a) We will work in $\mathop{{\rm End}}V$. Let $x\in \mathop{{\rm End}}V$ of order $p^k$. Since $(x-1)^{p^k}=x^{p^k}-1=0$, we get a descending chain of definable subgroups $$V\geqslant V(x-1)\geqslant V(x-1)^2\geqslant \cdots$$ which reaches $0$ in at most $p^k$ steps. Thus, the chain does not become stationary before it reaches $0$. Therefore,  $V(x-1)=[V,x]$ is a proper subgroup in $V$.\\
	(b) Since $V$ has \fmrd, for any $p$-element $x$ acting definably on $V$ the above chain reaches $0$ in at most $\rk(V)$ steps. Therefore, if $p^k\geqslant \rk(V)$ then $V(x^{p^k}-1)=V(x-1)^{p^k}=0$. Since $P$ is a $p$-torus, for any $y\in P$, there exists $x\in P$ such that $y=x^{p^k}$. Hence, $V(y-1)=V(x^{p^k}-1)=0$, and we are done.
	\end{proof}

\section{Preliminary Results}

Throughout this section, we assume $G$ and $V$  are groups of \fmrd,  $V$ is a connected elementary abelian $p$-group of Morley rank $n$, where $p$ is an odd prime, and $G$ acts on  $V$ definably and faithfully.

\begin{lemma} \label{2torus} Let $H$ a definable connected subgroup of $G$, and $q\neq p$ a prime number. Then $H$ does not contain any definable connected $q$-groups of bounded exponent. In particular, if $H$ has an involution, then the connected component of any of its Sylow $2$-subgroups is a $2$-torus.
\end{lemma}

\begin{proof} Combine Facts~\ref{ptori}, \ref{sylow2} and \ref{degenerate}.
\end{proof}

\subsection{Groups of $p$-unipotent type}

Following \cite{Borovik-Burdges-Cherlin2007}, we shall call a group $K$ a  \emph{group of $p$-unipotent type}, if  every definable connected solvable subgroup in $K$ is a nilpotent $p$-group of bounded exponent. We still work under the assumptions of this section.

\begin{proposition} \label{p-unipotent type}
Let $K$ be a definable subgroup in $G$ which contains no good tori. Then $K$ is a torsion group of $p$-unipotent type. In addition, $K$ does not contain non-trivial definable divisible abelian subgroups.
\end{proposition}

\begin{proof} First note that by Fact~\ref{trivialaction}, $K$ contains no nontrivial $p$-tori.
Therefore, every connected definable solvable $p^\perp$-subgroup in $K$ is trivial by Fact \ref{fact:good-torus}.

Now it is easy to see that very definable divisible abelian subgroup in  $K$ is trivial.
Indeed, if such a subgroup, say $A$, contains a non-trivial $p$-element, then it contains a non-trivial $p$-torus, which is impossible by the above paragraph. Hence $A$ is a $p^\perp$-group and is trivial again by above.

Next, notice that every element  in  $K$ is of finite order.
Indeed, if $x\in K$ is of infinite order, then the connected component $d(x)^\circ$ of the definable closure of $\langle x\rangle$ is a divisible abelian group, which contradicts the above paragraph.

By Fact \ref{structure-of-nilpotent}, $M=BD$, where $B$ and $D$ are connected,  $B$ is of bounded exponent and $D$ is divisible. However, $D=1$ by above,  and $B$ is a $p$-group by Lemma \ref{2torus}. Therefore, every  definable connected nilpotent subgroup $M$  in  $K$ is a $p$-group of bounded exponent.

By Fact~\ref {structure-solvable},
if  $M$ is a connected solvable subgroup  in  $K$, then $M = [M,M] C$ where $C$ is a connected nilpotent subgroup.

Finally, we will prove that $K$ is of $p$-unipotent type.
Let $M$ be a definable connected solvable subgroup of  $K$. Then by above, $M = [M,M] C$ where $C$ is a connected nilpotent subgroup.  By Fact~\ref{commutator-nilpotent},
 $[M, M]$ is also connected and nilpotent. Hence both subgroups are $p$-groups of bounded exponent by above; therefore, so is $M$. Now the nilpotency of $M$ follows from Fact~\ref{solvnilp}.
\end{proof}



\subsection{Basis of Induction}
Connected groups acting faithfully and definably on abelian groups of Morley rank $n\leqslant 3$ are well understood. To prove these special cases of our theorem, the following results will be used.

\begin{fact} {\rm (Deloro)} {\rm \cite{Deloro09}} \label{deloro} Let\/ $G$ be a connected non-solvable group acting faithfully on a connected abelian group $V$. If $\rk(V)=2$, then there exists an \acf $K$ such that the action $G\curvearrowright V$ is equivalent to $\operatorname{GL}_2(K)\curvearrowright K^2$ or $\operatorname{SL}_2(K)\curvearrowright K^2$.
\end{fact}

\begin{fact} {{\rm (Borovik--Deloro + Fr\'econ)} {\rm \cite{Borovik-Deloro}} and {\rm \cite{Frecon18}}.} \label{bordel} Let\/ $G$ be a connected non-solvable group acting faithfully and minimally on an abelian group $V\!$. If\/ $\rk(V)=3$  then there exists an \acf $K$ such that $V=K^3$ and $G$ is isomorphic to either\/ $\operatorname{PSL_2}(K)\times Z(G)$ or\/ $\operatorname{SL_3}(K)* Z(G)$. The action is the adjoint action in the former case, and the natural action in the latter case.
\end{fact}

\subsection{Throwback to pseudoreflection actions}

To exclude the case when $G$ in our Theorem \ref{maintheorem} is not connected, we will need a result which uses concepts from one of our earlier papers \cite{bbpseudo}. A special case of this result, when $G$ is connected, was stated as \cite[Corollary 1.3]{bbpseudo}.

\begin{proposition} \label{maximality-of-GL} Let $G$ be a group of finite Morley rank  acting definably and faithfully on an elementary abelian $p$--group $V$ of Morley rank $n$, where $p$ is an odd prime. Assume that  $G$ contains a definable subgroup $G^\sharp \simeq \mathop{\rm GL}_n(F)$ for an algebraically closed field $F$ of characteristic $p$. Assume also that $V$ is definably isomorphic to the additive group of the $F$-vector space  $F^n$ and $G^\sharp$ acts on $V$ as on its canonical module. Then  $G^\sharp =G$.
\end{proposition}

\begin{proof}
Observe first that  $\rk F^n = n$ implies $\rk F= 1$. \emph{Pseudoreflection subgroups} in the sense of \cite{bbpseudo} are connected  definable abelian subgroups $R <G^\sharp$ such that $V = [V,R] \oplus C_V(R)$ and $R$ acts transitively on the non-zero elements of $[V,R]$. By Fact~\ref{zilber}, one can immediately conclude that $R\simeq F^*$ and $[V,R]\simeq F^+$. Therefore $\rk R = 1=\rk [V,R]$ in our case.

It is easy to see that pseudoreflection subgroups in $G^\sharp = \mathop{\rm GL}_n(F)$ are $1$-dimensional (in the sense of the theory of algebraic groups) tori of the form, in a suitable coordinate system in $F^n$,
\[
R = \left\{\, \mathop{\rm diag}(x, 1,\dots, 1) \mid x  \in  F, x\ne 0\,\right\},
\]
and all pseudoreflection subgroups in $G^\sharp$  are conjugate in $G^\sharp$.

If $R$ is a pseudoreflection subgroup in $G^\sharp$, consider the subgroup $\langle R^G \rangle$ generated in $G$ by all $G$-conjugates of $R$, this is a normal definable subgroup generated by pseudoreflection subgroups, and, in view of \cite[Theorem 1.2]{bbpseudo},  $G^\sharp = \langle R^G \rangle$ is normal in $G$.

We will use induction on $n\geqslant 1$. When $n=1$,  $G^\sharp=R\simeq F^*$ and $V=[V,R]\simeq F^+$. The subring generated by $R$ in the ring of definable endomorphisms of $V$ is a field by Schur's Lemma, which we will denote by $E$. Since $G$ normalizes $R$, $G$ acts as a group of field automorphisms on $E$. Hence by Fact~\ref{automorphisms},  $G=C_G(R)$. Thus, $G$ acts linearly on $V\simeq F^+$ and therefore, $G=F^*=R=G^\sharp$.

Now assume $n\geqslant 2$. By the Frattini Argument, $G= G^\sharp N_G(R)$. Denote $H = N_G(R)$ and $H^\sharp = N_{G^\sharp}(R)$. It is well-known that $H^\sharp = R \times L$, where $L\simeq \mathop{\rm GL}_{n-1} (F)$ centralises $[V,R]$ and acts on $C_V(R) \simeq F^{n-1}$ as on a canonical module.  Obviously, $H$ leaves $[V,R]$ and $C_V(R)$ invariant. 
Note that the action of $H/R$ on $C_V(R)$ is faithful. Indeed, if $K/R$ is the kernel, then $K$ acts faithfully on $[V,R]$ with a normal subgroup $R\cong F^*$. This brings us to the base of induction which was discussed above. Hence $K=R$.   Thus,  $H/R$ contains a definable subgroup $H^\sharp/R=(R\times L)/R \simeq  \mathop{\rm GL}_{n-1} (F)$; by the inductive assumption, $H/R = H^\sharp/R$, hence $H =H^\sharp$. Therefore, $G= G^\sharp H=G^\sharp H^\sharp=G^\sharp$. \end{proof}

\section{Proof of Theorem \ref{maintheorem}}

We present a complete proof of Theorem \ref{maintheorem} in this section. Therefore, we work under the following assumptions.

We have a group of  \fmr $G$ acting definably, faithfully, and generically $m$-transitively on a connected elementary abelian $p$-group $V$ of  Morley  rank $n$, with $p$ an odd prime, and $m\geqslant n$.

First note that $V$ is connected by Fact~\ref{connectednessofV}. Another crucial observation is that $G^\circ$ also acts definably, faithfully, and generically $m$-transitively on $V$ by Fact~\ref{connectedG}. Therefore, in view of Proposition \ref{maximality-of-GL}, it will suffice to prove Theorem \ref{maintheorem} in the special case when $G=G^\circ$ is connected.

Therefore, from now on we assume that $G$ is connected.


\subsection{The core configuration}

We will focus now on a group-theoretic configuration at the heart of Theorem \ref{maintheorem}.

The generic $m$-transitivity of $G$ on $V$ means that there is a  generic subset  $A$  of $V^m$ on which $G$ acts transitively. We fix  an element $\bar{a}=(a_1,\dots,a_m)\in A$, and denote
\[
V_0 =  \langle a_1, \dots, a_m\rangle.
\]

From now on, we denote by   $K$ the pointwise stabilizer, and by $H$ the setwise stabilizer of $\{\pm a_1,\dots,\pm a_m\}$ in $G$.

In $\overline{H} = H/K$ we have $m$ involutions $\bar{e}_i$, $i=1,\dots, m$ defined by their action on $a_1,\dots, a_m$:
\[
a_j^{\bar{e}_i} = \left\{\begin{array}{rl}
    -a_j & \mbox{ if } i=j, \\
    a_j  & \mbox{ otherwise.}
    \end{array} \right.
\]

\begin{lemma} The group $\overline{E}_m = \langle\bar{e}_1,\dots,\bar{e}_m\rangle$ is an elementary abelian group of order $2^m$ and    $H/K\simeq  \Sigma_m= \overline{E}_m \rtimes {\rm Sym}_m$, where ${\rm Sym}_m$ permutes the generators $\bar{e}_1,\dots,\bar{e}_m$ of $\overline{E}_m$.
	\label{hyperoctahedral}
\end{lemma}

Notice that the group $\Sigma_m$ is the \emph{hyperoctahedral} group of degree $m$, which  prominently features in the theory of algebraic groups as
 the reflection group of type ${\rm BC}_m$. This fact is not used in this paper, but is likely to pop up in some of our future work.

\begin{proof}
Since $G$ acts generically $m$-transitively on $V$, the proof of \cite[Lemma~3.1]{bbsharp} can be repeated in this context as well, and we obtain the desired result.
\end{proof}

\subsection{Essential subgroups and ample subgroups}

Let $D$ be the full preimage in $H$ of the subgroup $\overline{E}_m < H/K$. At this point, we temporarily forget about the ambient group $G$ and generic transitivity and focus on the group $H$ and its subgroups $D$ and $K$.

For a subgroup $X \leqslant H$, we denote $X_D = X\cap D$ and $X_K = X \cap K$.

We shall call a definable subgroup $X \leqslant H$ \emph{ample}  if $KX=H$.

A definable subgroup $X \leqslant D$ is \emph{essential } if $KX = D$.
Equivalently, a definable subgroup  $X< G$  is essential if
\begin{itemize}
\item $X$ leaves invariant the set $ \{\, \pm a_1, \dots \pm a_m\,\}$ (which is equivalent to $X \leqslant H$), and, consequently, the subgroup $V_0$;
\item  $X_K$ is the pointwise stabiliser of $\bar{a}$ in $X$ (and consequently $X_K= C_X(V_0)$), and $X/K_X \simeq  \overline{E}_m $ acts on $V_0$ as on the canonical module $\mathbb{Z}_p^m$ for $\overline{E}_m$ and leaves invariant subgroups
    \[
    A_1=\langle a_1 \rangle, \dots, A_m=\langle a_m \rangle.
    \]
\end{itemize}

Notice that $X=H$ is an ample subgroup. Obviously, if $X$ is  ample, then $X_D$ is essential.

The following lemma summarises application of representation theory of finite groups in our context.

\begin{lemma} \label{key-Maschke}
If $X$ is a finite essential subgroup and $X_K$ is a $p'$-group, then $X_K =1$. Also, in that case $m=n$.
\end{lemma}

\begin{proof}
Since $X_K$ is a $p'$-group,  $X$ is also  a $p'$-group because $p\ne 2$ and $X/X_K$ is a $2$-group which covers $D/K = \overline{E}_m$.

Let $R$ be the enveloping algebra of $X$.

Notice that  $X$ (hence $R$) acts on each subgroup $A_1, \dots ,A_m$ irreducibly. Moreover, each representation is different, because the $A_i$'s are cyclic groups of order $p$, and, among the elements $\bar{e}_1,\dots,\bar{e}_m$, only $\bar{e}_i$ inverts $A_i$. By Theorem \ref{th:Maschke}, $m=n$ and $V = V_1\oplus \cdots\oplus V_n$, where in the modules $V_i$ each simple $R$-submodule is isomorphic to $A_i$. But $X_K$ acts trivially on each $A_i$, hence acts trivially on each $V_i$ and therefore on $V$. This means $X_K=1$.
\end{proof}

Notice that in the next lemma we do not assume that $X$ is finite, and therefore we continue to accept the possibility that $m > n$.

\begin{lemma} \label{lem:K-odd}
   If $X$ is an essential subgroup then $X_K$ is a $2^\perp$-group.
   \end{lemma}

   \begin{proof}
     Let $S$ be a Sylow $2$-subgroup in $X$. If $X_K$ is not a $2^\perp$-group, then $X_K\cap S \ne 1$. Take a non-trivial element $s \in X_K\cap S$ and elements $s_1,\dots, s_m$ in $S$ whose images in $X/X_K$ generate $X/X_K \simeq \overline{E}_m$. By Fact~\ref{sylow2}, $S$ is locally finite, hence $s, s_1,\dots,s_m$ generate a finite $2$-subgroup, say $Y$, in $S$. Obviously $X_KY=X$ hence $Y$ is essential and $s=1$ by Lemma \ref{key-Maschke}, a contradiction.
   \end{proof}

We can now characterise essential subgroups.

\begin{lemma} \label{essential-and Sylows} Let $E$ be a Sylow $2$-subgroup in $D$, then $KE=D$, $E_K=1$, and  $E\simeq \overline{E}_m$. In particular, $E$ is an essential subgroup.

Moreover, essential subgroups of $D$ are exactly those definable subgroups which contain one of the  Sylow $2$-subgroups of $D$.
\end{lemma}

\begin{proof} Since $E$ is a Sylow 2-subgroup in $D$, so is $KE/K$ in $D/K\simeq \overline{E}_m$, thus $D=KE$, that is, $E$ is essential. By Lemma~\ref{lem:K-odd}, $E_K$ is a $2^\perp$-group, hence it is trivial. Since $K/E_K\simeq\overline{E}_m$, the first statement follows. The second statement is clear.\end{proof}

Now we obtain the equality $m=n$ in the general case.

\begin{lemma} \label{m=n} We have
$m=n$.
\end{lemma}

\begin{proof} Since $H$ contains an elementary abelian 2-subgroup of order $2^m$ by Lemma~\ref{essential-and Sylows}, Fact~\ref{weightspaces} (or  Theorem \ref{th:Maschke}) gives us $m=n$.	\end{proof}

   \begin{lemma} \label{no-tori-in-K}
   $K$ contains no good tori.
   \end{lemma}

   \begin{proof}
     Assume the contrary, and let $T$ be a maximal good torus in $K$. By conjugacy of maximal good tori  \cite[Proposition IV.1.15]{abc}, and the Frattini Argument, we have $D = KN_D(T)$ and therefore $N_D(T)$ is an essential subgroup and contains a Sylow $2$-subgroup $E$ of $D$. Let $q$ be a prime such that $T$ has $q$-torsion and $Q$ the maximal elementary abelian $q$-subgroup in $T$. Obviously, $E$ normalises $Q$ and $QE$ is an essential subgroup. By Lemma \ref{key-Maschke}, $Q=1$, a contradiction.
   \end{proof}

\begin{lemma}
$K$ is a torsion group of $p$-unipotent type.
\end{lemma}

\begin{proof}
This is an immediate consequence of Lemma \ref{no-tori-in-K} and Proposition \ref{p-unipotent type}.
\end{proof}

    \begin{lemma}  \label{N(E)=1} Let $E$ be a Sylow $2$-subgroup in $D$. Then    $X=N_H(E) $ is an ample subgroup. Moreover, $X_K =1$ and $X\simeq  \Sigma_n$.
   \end{lemma}

   \begin{proof}
     By the Frattini Argument, $DX = H$, therefore $X$ is an ample subgroup. Since $X_K \lhd X$, $E \lhd X$, and $X_K\cap E=1$ by Lemma~\ref{lem:K-odd}, we have $[X_K, E] = 1$. If $x $ is a $p'$-element in $X_K$, then the subgroup $\langle x\rangle \times E$ is essential and therefore $x=1$ by Lemma \ref{key-Maschke}. Hence $X_K$ is a $p$-group.

     Take the weight decomposition of $V$ with respect to $E$:
     \[
     V = V_1 \oplus \cdots \oplus V_n, \mbox { where } i = 1,2,\dots, n.
     \]
    Note that every element in $X_K$ leaves every $1$-dimensional space $V_i$ invariant. Since $X_K$ is a $p$-group, Fact~\ref{trivialaction} is applicable, thus we conclude that $X_K$ fixes each $V_i$, and hence $V$, elementwise. Therefore, $X_K=1$, which, in its turn,  implies $X\simeq  \Sigma_n$.
   \end{proof}

   \subsection{An almost final configuration: the ample  subgroup $K^\circ \Sigma$}

\begin{lemma} If $m\geqslant 2$ then $G$ is not solvable. \label{notsolvable}
\end{lemma}

\begin{proof} If $G$ is solvable, by Fact~\ref{smoothlyirreducible}, we know that Fact~\ref{minimalaction} is applicable, hence  $G'=1$ and $G$ is abelian. However, for $m\geqslant 2$, $\Sigma_m$ is not abelian.
\end{proof}

At the heart  of our proof of  Theorem~\ref{maintheorem}, there is a core configuration.  We set it up by denoting $Q = K^\circ$ and $X =  QN_H(E)$ and listing the properties of $X$ which we have established so far.

\medskip

\begin{quote}
	\emph{The Core Configuration:}
	\bi
	\item $X$ is a group of finite Morley rank acting definably and faithfully on an elementary abelian $p$-group $V$, $p$ odd, of Morley rank $n\geqslant 3$.
	\item $Q\vartriangleleft X$ is a non-trivial connected definable subgroup of $p$-unipotent type; notice that $Q$ is not necessarily nilpotent. We also denote it $Q_n$.
	\item $\Sigma \simeq \Sigma_n$ is a subgroup of $X$ which normalises  $Q$. It will be convenient to denote it just $\Sigma_n$.
	\item Finally, to emphasise the inductive nature of our setting, we may write, if necessary, $X = X_n$.
	\ei
\end{quote}

In the next Lemma \ref{lemma:Q=1} we analyse the Core Configuration  on its own, without using any further information about $G$,  and prove that $Q=1$.


\begin{lemma}
\label{lemma:Q=1}
Under assumptions of  the Core Configuration, $Q_n=1$ for all $n \geqslant 3$.
\end{lemma}

\begin{proof} We proceed by induction on $n$. If $n = 3$,  $Q_3$ is nilpotent in view of Facts~\ref{deloro} and  \ref{bordel},  which give us a basis of induction on $n\geqslant 3$.   For the inductive step for $n >3$, take the involution $e_n$ and consider
\[
C_{Q\Sigma}(e_n) = C^\circ_Q(e_n)C_{\Sigma}(e_n).
\]
Obviously, this group leaves  invariant the eigenspaces $V^+_n$ and $V^-_n$.
Denote  $Q_{n-1} = C^\circ_Q(e_n)$. Observe   $C_{\Sigma}(e_n)=  \Sigma_{n-1} \times \langle e_n\rangle$, and we denote $X_{n-1} = Q_{n-1}\Sigma_{n-1}$.

For the inductive assumption, we have  $Q_{n-1}=1$. Then $C^\circ_Q(e_n) = Q_{n-1} =1$  and $Q_n$ is an abelian $p$-group. Assume $Q_n\neq 1$, then $[V, Q_n]$ is a non-trivial proper connected subgroup of $V$ and is $\Sigma_n$-invariant, which contradicts the minimality of the action of $\Sigma_n$ on $V$, Fact \ref{smoothlyirreducible}.  Contradiction shows $Q_n=1$.
\end{proof}

We can now return to the main proof.


\begin{proposition} \label{prop:almost final}
$K^\circ = 1$.
\end{proposition}

\begin{proof}
	If $n \leqslant 3$, we know everything about $G$ from Facts~\ref{deloro}, \ref{bordel}, and Lemma~\ref{notsolvable}, and  in these cases $K^\circ=1$. If  $n \geqslant 3$, 
	the proposition follows from Lemma~\ref{lemma:Q=1}.
\end{proof}

\begin{corollary} \label{K-is-finite}
 If $X$ is an ample subgroup then $X_K$ is a finite group without involutions and $X = X_K\rtimes \Sigma$.
\end{corollary}

\subsection{The final case: $K$ is finite}

\begin{lemma} \label{lem:K-is-1}
$K = 1$ and $H = \Sigma$.
 \end{lemma}

 \begin{proof}
Let $Q\ne 1$ be a Sylow $q$-subgroup of $K$ for a prime  $q \ne p$, then by the Frattini argument $H=KN_H(Q)$. Denote $X = N_H(Q)$, then $X$ is an ample subgroup. We can assume without loss of generality that $E <X$.  Applying Lemma \ref{key-Maschke} to the essential group $QE$, we see that $Q=1$.  Hence if $K \ne 1$ then  $K$ is a $p$-group. Consider the ample group $K\Sigma$; because of the minimality of the action of $\Sigma$ on $V$
   (Fact~\ref{smoothlyirreducible}), we have $K=1$ by Fact~\ref{trivialaction}(a). Hence $H=\Sigma$.
 \end{proof}

This completes the proof of Theorem \ref{maintheorem}.

\section*{Acknowledgements}

We thank Ali Nesin and staff and volunteers of the Nesin Mathematics Village for their hospitality; a lot of work on this and previous papers in our project was done there. We also thank the anonymous referee for reading our work carefully and making many valuable suggestions.

\end{document}